\documentclass[review]{elsarticle}

\usepackage{lineno}
\usepackage{hyperref}
\usepackage{amssymb,amsmath,amsthm,mathtools}
\usepackage{changepage}
\usepackage[english]{babel}
\usepackage[T1]{fontenc}
\usepackage{lmodern}
\usepackage{microtype}
\usepackage[textsize=tiny]{todonotes}
\usepackage{tikz}
\usetikzlibrary{graphs,graphs.standard,calc,math}
\usetikzlibrary{arrows, arrows.meta}
\usetikzlibrary{automata,arrows,positioning,calc}
\usepackage{subcaption}
\usepackage{graphicx}
\usepackage{caption}
\usepackage{pgfplots}
\usepackage{nicefrac}

\usepackage{enumitem}
\setlist{noitemsep,labelwidth=*,leftmargin=*,align=left}
\setlist[enumerate,1]{label=(\alph*)}
\setlist[description]{font=\normalfont,leftmargin=!}
\usepackage[nocomma, short]{optidef}

\usepackage[ruled,noend]{algorithm2e}
\renewcommand{\SetProgSty}[1]{\renewcommand{\ProgSty}[1]{\textnormal{\csname#1\endcsname{##1}}\unskip}}
\SetProgSty{}
\SetFuncSty{textsc}
\SetFuncArgSty{}
\SetArgSty{}
\SetCommentSty{emph}
\DontPrintSemicolon

\SetKwComment{tcp}{\textup{\texttt{\#}} }{}
\SetKwComment{tcc}{''' }{ '''}
\SetKwProg{Def}{def}{:}{}

\newcommand{\algass}{\ensuremath{\leftarrow}}

\renewcommand{\O}[1]{\ensuremath{\mathcal{O}(#1)}}

\SetKw{Continue}{continue}
\SetKw{Break}{break}

\SetKwInput{Input}{Input}
\SetKwInput{Output}{Output}
\SetKw{Return}{\textbf{Return:}}
\SetKw{And}{\textbf{and}}
\SetKwComment{Comment}{/* }{ */}

\usepackage[capitalise]{cleveref}

\theoremstyle{definition}
\newtheorem{theorem}{Theorem}[section]

\newtheorem{corollary}[theorem]{Corollary}
\newtheorem{definition}[theorem]{Definition}
\newtheorem{remark}[theorem]{Remark}
\newtheorem{lemma}[theorem]{Lemma}
\newtheorem{problem}[theorem]{Problem}

\newtheorem{notation}[theorem]{Notation}
\newtheorem{assumption}[theorem]{Assumption}
\theoremstyle{definition}

\newcommand{\M}{\ensuremath{\mathcal{M}}}

\newcommand{\F}{\ensuremath{\mathcal{F}}}
\newcommand{\I}{\ensuremath{\Lambda}}

\newcommand{\Set}[1]{\left\{ #1 \right\}}
\newcommand{\R}{\ensuremath{\mathbb{R}}}
\newcommand{\N}{\ensuremath{\mathbb{N}}}
\newcommand{\inn}{\ensuremath{\operatorname{int}}}
\newcommand{\ver}{\ensuremath{\operatorname{vert}}}

\newcommand{\appri}{\mathcal{A}}
\newcommand{\oracle}{\mathsf{ORC}}
\newcommand{\alg}{\mathsf{ALG}}
\newcommand{\cell}{c}

\pgfplotsset{compat=1.18}

\journal{Journal of ...}

\biboptions{authoryear}

\begin{document}
	
\begin{frontmatter}
	
	\title{An Approximation Framework for Parametric Matroid Interdiction Problems}
	
	\author[a]{Nils Hausbrandt}\corref{mycorrespondingauthor}
	\cortext[mycorrespondingauthor]{Corresponding author}
	\ead{nils.hausbrandt@math.rptu.de}
	\author[a]{Levin Nemesch}
	\ead{l.nemesch@math.rptu.de}
	\author[a]{Stefan Ruzika}
	\ead{stefan.ruzika@math.rptu.de}
	
	\address[a]{Department of Mathematics, RPTU University Kaiserslautern-Landau, 67663 Kaiserslautern, Germany}
	
	\begin{abstract}
		Matroid interdiction problems are well-researched in the field of combinatorial optimization.
		In the matroid $\ell$-interdiction problem, an interdiction strategy removes a subset of cardinality $\ell$ from the matroid's ground set.
		The goal is to maximize the weight of a remaining optimal basis.
		We examine the multi-parametric generalization of this problem, where every weight is given by a linear function depending on a parameter vector.
		For every parameter value, we are interested in an optimal interdiction strategy and the weight of an optimally interdicted basis.
		We develop the first framework for lifting approximation algorithms for the non-parametric matroid $\ell$-interdiction problem to its multi-parametric variant.
		Whenever there exists a $\beta$-approximation algorithm for the non-parametric problem, we obtain an approximation algorithm for the multi-parametric problem with an approximation quality arbitrarily close to $\beta$.
		Our method yields an FPTAS for partition matroids and a $(1-\varepsilon)\frac{1}{4}$-approximation for graphic matroids.
		As part of the construction, we develop the first approximation algorithm for a conventional multi-parametric optimization problem in which the parameter vector varies in an arbitrary polytope.    
	\end{abstract}
	
	
	\end{frontmatter}

\section{Introduction}\label{sec:intro_p_l_AX}
	In this article, we focus on the interdiction of matroids -- one of the fundamental structures of mathematics that generalizes many well-known concepts, including independence in linear algebra, cf.\ \cite{whitney1935on, wilson1973introduction, welsh2010matroid, oxley2011matroid}.
	In the matroid interdiction problem, an interdiction budget is given and each element of the matroid's ground set is associated with a weight and an interdiction cost.
	The goal is to remove a subset of the elements whose interdiction costs meet the budget while maximizing the weight of a minimum weight basis.
	
	Interdiction problems offer a highly interesting field of research in combinatorial optimization and optimization research, cf.\ \cite{smith2013modern, smith2020survey}.
	The problems can be seen as a game between two players, called \emph{interdictor} and \emph{follower}.
	The interdictor tries to worsen the follower's objective function value as much as possible.
	The interdictor's actions, called interdiction strategies, are limited by interdiction costs and an interdiction budget.
	
	Particularly in the last decade, matroid interdiction problems have gained a lot of interest.
	\cite{frederickson1998algorithms} measured the perturbability of a matroid, \cite{adjiashvili2015bulk} introduced bulk robustness,
	\cite{joret2015reducing} considered the rank interdiction problem, \cite{chestnut2017interdicting} developed pseudoapproximations, \cite{wei2024supervalid} constructed supervalid inequalities, \cite{ketkov2024class} focused on a special class of partition matroids, \cite{weninger2024interdiction} provide a dynamic programming approach and \cite{hausbrandt2024parametric, hausbrandt2024theparametric, hausbrandt2025multi} extended to (multi-)parametric variants.
	Note that there is a lot of literature on the special case of minimum spanning tree interdiction, see e.g. \cite{frederickson1999increasing,bazgan2013critical,zenklusen2015approximation}, and that we focus on arbitrary matroids here.
	
	If all interdiction costs are equal to one and the budget equals a fixed natural number $\ell \in \N_{>0}$, the matroid interdiction problem can be solved exactly in polynomial time if an independence test can be performed in time polynomial in the input length, cf.\ \cite{hausbrandt2024theparametric, weninger2024interdiction}.
	However, if $\ell$ is considered as part of the input of the problem, the problem is $\mathsf{NP}$-hard - even for graphic matroids, see \cite{frederickson1999increasing}.
	For the two special cases of uniform matroids and partition matroids, it is known that this variant can be solved in polynomial time, cf.\ \cite{weninger2024interdiction}.
	For the general variant with arbitrary interdiction costs and budget, the matroid interdiction problem is $\mathsf{NP}$-hard for uniform, partition and graphic matroids, cf.\ \cite{frederickson1999increasing, weninger2024interdiction}.
	For graphic matroids, \cite{zenklusen2015approximation} constructed an approximation algorithm with constant factor guarantee, which was improved by \cite{linhares2017improved} to a $\nicefrac{1}{4}$-approximation.
	
	In this article, we consider a multi-parametric variant of the matroid interdiction problem where the weights of the elements of the matroid's ground set depend linearly on a fixed $p\in\N_{>0}$ of parameters coming from a parameter polytope $\I\subseteq\R^p$, which is defined by an intersection of half-spaces.
	The goal is to compute, for each possible parameter vector $\lambda\in\I$, an optimal interdiction strategy and the weight of the optimal interdicted minimum weight basis.
	An optimal interdiction strategy consists of a set of $\ell\in\N_{>0}$ elements whose removal increases the weight of the minimum weight basis as much as possible.
	
	Parametric optimization is an active and growing field of research since at least 1970, a comprehensive survey is provided by \cite{nemesch2025survey}.
	In a general parametric optimization problem, the objective function value of a feasible solution depends (often linearly) on a parameter vector.
	Parametric variants exist for many well-known combinatorial optimization problems, among them minimum spanning trees and more general matroid problems, cf.\ \cite{eppstein1995geometric,fernandez1996using,agarwal1998parametric}.
	
	The combination of parametric optimization and interdiction of matroids is a recent development.
	The literature so far considers the special cases where the interdiction budget equals one or a fixed natural number $\ell$, cf.\ \cite{hausbrandt2024parametric, hausbrandt2024theparametric, hausbrandt2025multi}.
	We build upon this literature and extend it by investigating the multi-parametric matroid $\ell$-interdiction problem where all interdiction costs are equal to one and the budget is given by a natural number $\ell\in\N_{>0}$ which is part of the input of the problem.
	
	To assess the complexity of the problem, it is necessary to consider intractability, approximability, and the complexity of the corresponding non-parametric problem.
	As mentioned above, the non-parametric problem, i.~e.\ the problem in which the parameter is fixed to a vector $\lambda\in\I$, is $\mathsf{NP}$-hard.
	It is unknown whether the multi-parametric matroid $\ell$-interdiction problem is intractable or not, i.~e.\, whether it requires a super-polynomial number of optimal interdiction strategies when the parameter vector varies within the parameter polytope $\I$.
	Our goal is to provide a significant step toward understanding the complexity of the problem. 
	We show that approximation algorithms for the corresponding non-parametric problem can be lifted to the multi-parametric problem.
	To achieve this, we combine structural results for matroid $\ell$-interdiction problems with a new multi-parametric approximation algorithm.
	As a result, we provide the first approximation algorithm for a (multi-)parametric interdiction problem.
	
	\medskip
	More precisely, we show that whenever there exists a polynomial time $\beta$-approximation algorithm for the non-parametric matroid $\ell$-interdiction problem ($0<\beta\leq1$), there exists a polynomial time approximation algorithm for the multi-parametric variant with approximation guarantee arbitrary close to $\beta$.
	More specifically, we assume that there exists an oracle $\oracle_\beta$ that computes, for a fixed parameter vector $\lambda\in\I$, a $\beta$-approximate interdiction strategy at $\lambda$.
	This is a subset of the matroid's ground set that contains $\ell$ elements whose removal results in a minimum weight basis with a weight that is at most a factor of $\beta$ smaller than the weight of an optimal $\ell$-interdicted minimum weight basis.
	Two prominent matroid interdiction problems that satisfy the assumption are the \emph{partition} as well as the \emph{graphic} matroid interdiction problem.
	Given the oracle $\oracle_\beta$, we develop an algorithm that computes, for a given $0<\varepsilon<1$, a $(1-\varepsilon)\beta$-approximation on the entire parameter polytope $\I$.
	Such an approximation consists of a decomposition of the parameter polytope into smaller polytopes and one interdiction strategy per polytope which is $(1-\varepsilon)\beta$-approximate for all vectors within the polytope.
	To lift the approximation guarantee $\beta$ from a fixed point $\lambda\in\I$ to a guarantee of $(1-\varepsilon)\beta$ on the entire polytope, we need to overcome two major hurdles.
	\begin{enumerate}
		\item First, we need to decompose the parameter polytope into polynomially many smaller polytopes on which the max-min structure of the multi-parametric interdiction problem is reduced to a common \footnote{This refers to a parametric problem of the form $\max\Set{c(x)+\lambda^\top d(x)\colon x\in X}$ with continuous and convex optimal value function, as considered, for example, in \cite{bazgan2022approximation}.}
		multi-parametric maximization problem.
		\item Second, we need to develop an approximation algorithm for a common multi-parametric maximization problem on a parameter polytope.
	\end{enumerate}
	Note that there already exist some approximation algorithms for common multi-parametric  maximization (minimization) problems, see~\cite{helfrich2022approximation} and~\cite{helfrich.ruzika.ea2024EfficientlyConstructing}.
	However, both algorithms assume a specific structure of the parameter set.
	It is unclear how they can be generalized to work on arbitrary parameter polytopes.
	Our algorithm closes this gap and works on arbitrary parameter polytopes, not only for matroid problems, but for a broad class of common multi-parametric maximization and minimization problems.
	Overall, our procedure yields an FPTAS for multi-parametric partition matroid $\ell$-interdiction and a $(1-\varepsilon)\frac{1}{4}$-approximation for multi-parametric minimum spanning tree $\ell$-interdiction.
	
	The remainder of this article is organized as follows.
	The preliminaries are provided in Section~\ref{sec:prelims_p_l}.
	In Section~\ref{sec:framework_p_l}, we develop an approximation framework for the multi-parametric matroid $\ell$-interdiction problem given the following two conditions.
	We assume the existence of an oracle $\oracle_\beta$ that computes a $\beta$-approximation for the non-parametric matroid $\ell$-interdiction problem at a fixed vector $\lambda\in\I$.
	Further, we assume that there exists, for any $0<\varepsilon<1$, a $(1-\varepsilon)\beta$-approximation algorithm $\alg_\varepsilon$ for a common multi-parametric maximization problem on a parameter polytope.\footnote{The algorithm $\alg_\varepsilon$ works on sub-polytopes of $\I$ on which the interdiction problem reduces to a maximization problem and, thereby, uses the oracle $\oracle_\beta$ itself as a black box for fixed vectors $\lambda\in\I$.}
	Given $\oracle_\beta$ and $\alg_\varepsilon$, we construct a $(1-\varepsilon)\beta$-approximation algorithm for the multi-parametric matroid $\ell$-interdiction problem on the entire parameter polytope $\I$.
	In Section~\ref{sec:approx_p_l}, we develop the approximation algorithm $\alg_\varepsilon$ for common multi-parametric maximization and minimization problems on an arbitrary parameter polytope.
	In Section~\ref{sec:application_p_l}, we apply our procedure to the prominent classes of partition and graphic matroids and specify the oracle $\oracle_\beta$ for these two cases.

\section{Preliminaries}\label{sec:prelims_p_l}
	In this section, we introduce the required definitions and specify the assumptions that apply throughout the article.
	
	For a singleton $\Set{s}$ and a set $\mathcal{S}$, we use the notation $\mathcal{S}-s$ or $\mathcal{S}+s$ instead of $\mathcal{S}\setminus\Set{s}$ or $\mathcal{S}\cup\Set{s}$.
	Further, we denote the relative interior of a set $\mathcal{S}$ by $\inn(\mathcal{S})$.
	
	\textbf{Matroids.}
	Let $E$ be a finite set called the ground set of the matroid and $\emptyset\neq\F\subseteq 2^E$ be a nonempty subset of the power set of $E$.
	The tuple $\M=(E,\F)$ is called \emph{matroid} if the following properties hold:
	\begin{enumerate}
		\item The empty set $\emptyset$ is contained in $\F$.
		\item If $A\in\F$ and $B\subseteq A$, then also $B\in\F$.
		\item If $A,B \in\F$ and $\vert B\vert < \vert A\vert$, then there exists an element $a\in A\setminus B$ such that $B + a\in\F$.
	\end{enumerate}
	The subsets of $E$ contained in $\F$ are called \emph{independent} and all other subsets of $E$ are called \emph{dependent}.
	An inclusion-wise maximal independent set is called \emph{basis}.
	All bases have the same cardinality, which is called the \emph{rank} $k\coloneqq rk(\M)$ of the matroid.
	We define the cardinality of $E$ by $\vert E\vert\coloneqq m$.
	
	\textbf{Multi-parametric matroids.}
	In this article, we consider multi-parametric matroids where the weights of the elements of the matroid's ground set depend linearly on a fixed number $p\geq1$ of parameters.
	To be more precise, the weight $w(e,\lambda)$ of an element $e\in E$ is given by an affine linear function
	$$w(e,\lambda)\coloneqq w(e,\lambda_1,\dotsc,\lambda_p)= a_e + \lambda_1 b_{1,e} +\dotsc+\lambda_p b_{p,e},$$ where $a_e, b_{1,e},\dotsc, b_{p,e}\in \mathbb{Q}$ are rational numbers.
	The parameter vector $\lambda=(\lambda_1,\dotsc,\lambda_p)$ is taken from a \emph{parameter polytope} $\I\subseteq\R^p$ which is defined by $q$ half-spaces.
	The multi-parametric weight of a basis $B$ at a point $\lambda\in\I$ is given by $w(B,\lambda)\coloneqq w(B,\lambda_1,\dotsc,\lambda_p)\coloneqq\sum_{e\in B} w(e,\lambda)$.
	In the \emph{multi-parametric matroid problem}, the goal is to compute, for each possible parameter vector $\lambda\in\I$, a minimum weight basis $B_\lambda^*$, that is, a basis $B^*$ which satisfies $w(B^*,\lambda)\leq w(B,\lambda)$ for all bases $B$ of $\M$, see \cite{hausbrandt2025multi}.
	
	To determine how the sorting of the elements $e\in E$ behaves over the parameter polytope $\I$, we consider the separating hyperplanes that arise when two weights $w(e,\lambda)$ and $w(f,\lambda)$ are intersected. 
	\begin{definition}[Separating hyperplane (\cite{seipp2013adjacency})]
		For two elements $e,f\in E$ with weight functions $w(e,\lambda)$ and $w(f,\lambda)$, the \emph{separating hyperplane} $h(e,f)\subseteq\R^p$ is defined as $h(e,f)~=~\Set{\lambda\in\mathbb{R}^p\colon w(e,\lambda)=w(f,\lambda)}$. It divides $\mathbb{R}^p$ into two open half-spaces
		\begin{enumerate}
			\item[] $h^<(e,f)=\Set{\lambda\in\mathbb{R}^p\colon w(e,\lambda)<w(f,\lambda)}$ and
			\item[] $h^>(e,f)=\Set{\lambda\in\mathbb{R}^p\colon w(e,\lambda)>w(f,\lambda)}$,
		\end{enumerate}
		where the former consists of all parameter vectors for which $e$ has a smaller weight than $f$ and the latter, conversely, contains all parameter vectors for which $f$ has a smaller weight than $e$.
		The set of all separating hyperplanes is denoted by $H^=\coloneqq\Set{h(e,f)\colon e,f\in E}$.
	\end{definition}
	The separating hyperplanes in $H^=$ define a so-called arrangement of hyperplanes, see \cite{edelsbrunner1987algorithms,orlik2013arrangements}.
	\begin{definition}[Arrangement]
		The \emph{arrangement} $A(H^=)$ consists of faces $s$ with $$s=\bigcap\limits_{h(e,f)\in H^=} t(h(e,f)),$$ where $t(h(e,f))\in\Set{h(e,f),h^<(e,f),h^>(e,f)}$.
		A full-dimensional face $s$ is called a \emph{cell} of $A(H^=)$.
	\end{definition}
	Each cell of the arrangement $A(H^=)$ corresponds to a unique basis that is minimal for all parameter vectors within the cell, cf.\ \cite{hausbrandt2025multi}.
	
	\textbf{Interdicting multi-parametric matroids.}
	In this article, we are allowed to interdict $\ell\in\mathbb{N}_{>0}$ elements of a matroid with multi-parametric weights.
	The following definitions as well as the problem formulation are an extension of the concepts in \cite{hausbrandt2024theparametric} ($p=1,\ell>1$) and \cite{hausbrandt2025multi} ($p>1,\ell=1$) to the case that $p>1$ many parameters are considered and $\ell\geq1$ many elements are allowed to be interdicted.
	\begin{notation}
		In this article, we consider a matroid $\M=(E,\F)$ with multi-parametric weights $w(e,\lambda)=a_e + \lambda_1 b_{1,e} +\dotsc+\lambda_p b_{p,e}$ for $\lambda\in\I\subseteq \R^p$, where
		\begin{enumerate}
			\item[] $m$ is the cardinality of $E$,
			\item[] $k$ is the rank of $\M$,
			\item[] $\ell$ is the number of elements allowed to be interdicted,
			\item[] $p$ is the number of parameters, and
			\item[] $q$ is the number of half-spaces defining $\I$. 
		\end{enumerate}   
	\end{notation}
	We define the following restricted matroid.
	We denote the matroid $(E',\F')$ with $E'\subseteq E$ and $\F'\coloneqq\Set{F\in\F\colon\, F\subseteq E'}$ by $\M|E'$.
	For a subset $F\subseteq E$, we use the shortcut $\M_F$ for $\M|(E\setminus F)$.
	An interdiction strategy consists of a subset $F\subseteq E$ containing $\ell$ elements.
	\begin{definition}[Set of $\ell$-most vital elements]
		Let $\lambda\in\I$.
		For a subset $F\subseteq E$, we denote a minimum weight basis on $\M_F$ at $\lambda$ by $B_\lambda^F$.
		If $\M_F$ does not have a basis of rank $k$, we set $w(B_\lambda^F,\lambda)=\infty$ for all $\lambda\in\I$.
		A subset $F^*\subseteq E$ with $\vert F^*\vert=\ell$ is called a \emph{set of $\ell$-most vital elements} at $\lambda$ if $w(B_\lambda^{F^*},\lambda) \geq w(B_\lambda^F,\lambda)$ for all $F\subseteq E$ with $\vert F\vert=\ell$.    
	\end{definition}
	The function mapping the parameters to the weight of an optimally $\ell$-interdicted basis $B_\lambda^{F^*}$ is called the optimal interdiction value function $y$.
	\begin{definition}[Optimal interdiction value function]
		For a subset $F\subseteq E$ with $\vert F\vert=\ell$, we define the function $y_F$ via $y_F\colon\I\to\mathbb{R}$, $\lambda\mapsto w(B_\lambda^F,\lambda)$ mapping the parameter vector $\lambda$ to the weight of a minimum weight basis of $\M_F$ at $\lambda$.
		For $\lambda\in\I$, we define $y(\lambda)\coloneqq\max\Set{y_F(\lambda)\colon\, F\subseteq E,\;\vert F\vert=\ell}$ as the weight of an optimally $\ell$-interdicted basis at $\lambda$.
		The \emph{optimal interdiction value function} $y$ is then defined via 
		$y\colon\I\to\mathbb{R}$, $\lambda\mapsto y(\lambda)$.
	\end{definition}
	The \emph{multi-parametric matroid $\ell$-interdiction problem} aims at obtaining an optimal interdiction strategy together with the weight of an optimally interdicted basis for each $\lambda\in\I$.
	More formally:
	\begin{problem}[Multi-parametric matroid $\ell$-interdiction problem]\label{prob:p_l_AX}
		Given a matroid~$\M$ with multi-parametric weights $w(e,\lambda)$, a parameter polytope~$\I\subseteq\R^p$ defined by $q$ half-spaces, and a number $\ell\in\mathbb{N}_{>0}$, the goal is to determine, for each $\lambda\in\I$, a set of $\ell$-most vital elements $F^*$ and the corresponding objective function value $y(\lambda) = y_{F^*}(\lambda)$.
	\end{problem}
	The problem obtained when fixing the parameter vector to $\lambda\in\I$ is called the \emph{non-parametric} (matroid $\ell$-interdiction) problem at $\lambda$.
	\begin{remark}
		In Problem~\ref{prob:p_l_AX}, the values $m,k,q$ and $\ell$ are considered to be part of the input of the problem, but the number $p$ of parameters is fixed.%
		\footnote{
			Fixing the number of parameters is common in multi-parametric approximation, cf.\ \cite{helfrich2022approximation}.
		}
		For a fixed parameter vector $\lambda \in\I$, the corresponding non-parametric matroid $\ell$-interdiction problem is $\mathsf{NP}$-hard.
		This follows from the $\mathsf{NP}$-hardness of the special case of the $\ell$-most vital edges problem with respect to graphic matroids, cf.\ \cite{frederickson1999increasing}.
		It is an open problem, whether Problem~\ref{prob:p_l_AX} is intractable, i.~e.\ if the number of interdiction strategies required to state an optimal interdiction strategy for each parameter vector is super-polynomial in $m,k$ or $\ell$.
		Furthermore, it generally seems to be more complicated to approximate a parametric interdiction problem such as Problem~\ref{prob:p_l_AX} than a common parametric problem, since the objective function $y$ is generally neither convex nor concave (\cite{hausbrandt2025multi}), which is a very useful property when computing parametric approximations, cf.\ \cite{bazgan2022approximation,helfrich2022approximation}.
		The goal of this article is to take a first step toward assessing the complexity of the problem by showing that we can at least compute approximations in polynomial time.
	\end{remark}
	We state our assumptions.
	\begin{assumption}\label{ass:p_l_AX} We assume the following.
		\begin{enumerate}
			\item There exists a basis $B_\lambda^F$ of cardinality $k$ for every $\lambda\in\I$ and $F\subseteq E$ with $\vert F\vert=\ell$.\label{ass:a_1_l_AX}
			\item No hyperplane $h(e,f)\in H^=$ is vertical, meaning that it contains a line parallel to the $p$-th coordinate axis.\label{ass:c_1_l_AX}
			\item For each $e\in E$, it holds $a_e,b_{1,e},\dotsc,b_{p,e}\geq 0$, where at most $k-1$ values $a_e$ and $k-1$ values $b_{i,p}$ (for all $i=1,\dotsc,p$) are equal to zero. \label{ass:d_1_l_AX}
			\item The parameter polytope~$\I$ is non-negative and compact. \label{ass:e_1_l_AX}
		\end{enumerate}
	\end{assumption}
	Assumption \ref{ass:a_1_l_AX} can be made without loss of generality, cf.\ \cite{hausbrandt2024theparametric}.
	It excludes the trivial case that no basis with rank $k$ exists after interdicting $\ell$ elements.
	Assumption \ref{ass:c_1_l_AX} is a technical assumption for the algorithm of \cite{edelsbrunner1986constructing} to compute the arrangement of the separating hyperplanes.
	The assumption can be ensured without loss of generality by common perturbation strategies (cf.\ \cite{edelsbrunner1987algorithms}).
	Once the cells of the arrangement $A(H^=)$ have been computed, the transformation can be shifted back.
	Assumptions \ref{ass:d_1_l_AX} and \ref{ass:e_1_l_AX} ensure that all functions $y_F$ and the objective function $y$ are positive on the entire parameter polytope $\I$, which is a standard assumption when considering approximation algorithms, cf.\ \cite{williamson2011design, bazgan2022approximation, helfrich2022approximation}.
	Note that each function value $y_F(\lambda)$ is the sum of exactly $k$ weights $w(e,\lambda)$.
	
	With Assumption~\ref{ass:p_l_AX}, we obtain a well-defined notion of approximation.
	The following two definitions are a direct extension of the concept of approximation for parametric optimization problems of \cite{bazgan2022approximation} and \cite{giudici2017approximation} to the parametric interdiction problem considered in this article.
	\begin{definition}
		Let $\lambda\in\I$ be fixed and $0<\alpha\leq1$.
		Let $F^*$ be a set of $\ell$-most vital elements at $\lambda$, i.~e.\ $y_{F^*}(\lambda)=y(\lambda)$.
		A subset $F\subseteq E$ with $\vert F\vert=\ell$ is called an \emph{$\alpha$-approximate set of $\ell$-most vital elements} at $\lambda$ if $y_F(\lambda)\geq\alpha\cdot y(\lambda)$.
		For an instance $\Pi=(\M,w,\ell,\I)$ of Problem~\ref{prob:p_l_AX}, a finite set $S$ of subsets $F\subseteq E$ containing $\ell$ elements is called an \emph{$\alpha$-approximation} if, for any $\lambda\in\I$, there exists an $\alpha$-approximate set $F\in S$ of $\ell$-most vital elements at $\lambda$.
	\end{definition}
	Note that an $\alpha$-approximation for Problem~\ref{prob:p_l_AX} is always finite by construction.
	\begin{definition}
		An algorithm $A$ that computes, for each instance of Problem~\ref{prob:p_l_AX}, an $\alpha$-approximation in time polynomial in the encoding length of the instance is called an \emph{$\alpha$-approximation algorithm} (for Problem~\ref{prob:p_l_AX}).
		A \emph{polynomial time approximation scheme (PTAS)} (for Problem~\ref{prob:p_l_AX}) is a family $(A_\varepsilon)_{\varepsilon>0}$ of algorithms such that, for every $0<\varepsilon<1$, algorithm $A_\varepsilon$ is a $(1-\varepsilon)$-approximation algorithm.
		A PTAS $(A_\varepsilon)_{\varepsilon>0}$ is called a \emph{fully polynomial time approximation scheme (FPTAS)} if the running time of each $A_\varepsilon$ is additionally polynomial in $\nicefrac{1}{\varepsilon}$.
	\end{definition}

\section{The approximation framework}\label{sec:framework_p_l}
	In this section, we construct a framework for an approximation algorithm for the multi-parametric matroid $\ell$-interdiction problem.
	One of the main hurdles to compute approximations for this problem is that the max-min structure of the problem implies that the objective function $y$ is in general neither convex nor concave, cf.\ \cite{hausbrandt2024theparametric, hausbrandt2025multi}.
	To overcome this obstacle, we show that the parameter polytope can be decomposed into polynomially many polytopes on each of which $y$ is given by a multi-parametric maximization problem.
	We consider each polytope separately and use the resulting convexity of the objective function $y$ to obtain a $(1-\varepsilon)\beta$-approximation for the multi-parametric matroid $\ell$-interdiction problem.
	
	In this section, we require the existence of the following two algorithms.
	First, we assume that there exists an oracle $\oracle_\beta$ with $0<\beta\leq1$ that computes a $\beta$-approximation for the corresponding non-parametric problem, see Section~\ref{sec:application_p_l} for concrete examples.
	Second, we require an algorithm $\alg_\varepsilon$ that computes a $(1-\varepsilon)\beta$-approximation for a common multi-parametric maximization problem on a parameter polytope.
	The algorithm $\alg_\varepsilon$ that uses the oracle $\oracle_\beta$ itself as a black box 
	constructed in Section~\ref{sec:approx_p_l}.
	
	Our procedure works as follows.
	We first compute the set $H^=$ of separating hyperplanes and add the hyperplanes defining the parameter polytope $\I$ to $H^=$.
	Afterwards, we compute the arrangement $A(H^=)$ of separating hyperplanes with the algorithm of \cite{edelsbrunner1986constructing}\footnote{For dimensions $d\geq3$, the proof of correctness originally presented in \cite{edelsbrunner1986constructing} contains an error that has been corrected by \cite{edelsbrunner1993zone}.}.
	We then use the decomposition of the parameter polytope into the cells of $A(H^=)$ and consider each cell $c\in A(H^=)$ separately.
	As shown in \cite{hausbrandt2025multi}, each cell $c\in A(H^=)$ can be assigned a unique basis that is optimal for all parameter vectors $\lambda\in c$.
	Similarly, we obtain the following lemma.
	\begin{lemma}\label{lem:y_convex_on_cell_p_l}
		On a cell $c\in A(H^=)$, each function $y_F$ is linear and, hence, their upper envelope is convex.
	\end{lemma}
	\begin{proof}
		The sorting of the weights $w(e,\lambda)$ does not change on a cell $c$.
		Consequently, each basis $B_\lambda^F$ remains unchanged on $c$, in the sense that if $B=B_\lambda^F$ for one $\lambda$ from the relative interior of $c$, then $B=B_\lambda^F$ for all $\lambda\in c$.
		Hence, each function $y_F(\lambda)=w(B^F_\lambda,\lambda)$ is linear on $c$ and $y$ is convex as a point-wise maximum of linear functions.
	\end{proof}
	We now formulate our approximation framework for the multi-parametric matroid $\ell$-interdiction problem in Algorithm~\ref{alg:framework_p_l}.
	Asymptotic running time and correctness are discussed in Theorem~\ref{theo:framework_p_l}.
	
	\begin{algorithm}[t]
		\Input{An instance $\Pi=(\M,w,\ell,\I)$ of Problem~\ref{prob:p_l_AX} and $\varepsilon>0$.} 
		\Output{A $((1-\varepsilon)\beta)$-approximation for $\Pi$.}
		
		Compute the set of separating hyperplanes $H^=$\;
		Add all facet supporting hyperplanes of $\I$ to $H^=$\;
		Compute the arrangement of separating hyperplanes $A(H^=)$\;
		Initialize $S\algass\emptyset$\;
		\For{each cell $c\in A(H^=)$}{
			Compute a $((1-\varepsilon)\beta)$-approximation $S_c$ on $c$ with $ALG_\varepsilon$\;
			Set $S\algass S\cup S_c$\;
		}
		\Return $S$\;
		\caption{An approximation framework for Problem~\ref{prob:p_l_AX}}
		\label{alg:framework_p_l}
	\end{algorithm}
	
	\begin{theorem}\label{theo:framework_p_l}
		Let $0<\varepsilon<1$ and a $(1-\varepsilon)\beta$-approximation algorithm $\alg_\varepsilon$ 
		for a common multi-parametric maximization problem on a parameter polytope be given.
		For an instance $\Pi=(\M,w,\ell,\I)$ of the multi-parametric matroid $\ell$-interdiction problem, Algorithm~\ref{alg:framework_p_l} computes a $(1-\varepsilon)\beta$-approximation for $\Pi$ in $\O{{(m^2+q)}^{p}\cdot T_{\alg_\varepsilon}}$ time, where $T_{\alg_\varepsilon}$ denotes the running time of $\alg_\varepsilon$.
	\end{theorem}
	\begin{proof}
		Let $c\in A(H^=)$ be a cell and $\lambda_0\in\inn(c)$ be a point from the relative interior of $c$.
		Lemma~\ref{lem:y_convex_on_cell_p_l} implies that, on $c$, Problem~\ref{prob:p_l_AX} is a common multi-parametric maximization problem with piece-wise linear, continuous and convex objective function of the form
		$$\max_{\substack{F\subseteq E, \vert F\vert=\ell \\ \lambda\in c}} y_F(\lambda)=w(B_{\lambda_0}^F,\lambda)$$
		as each function $y_F$ is linear on $c$.
		Consequently, we can apply the algorithm $\alg_\varepsilon$ on each of the polynomially many cells $c\in A(H^=)$ and obtain a $(1-\varepsilon)\beta$-approximation on the entire parameter polytope $\I$.
		
		We now analyze the asymptotic running time.
		The $\binom{m}{2}$ many separating hyperplanes in $H^=$ can be computed in $\O{m^2}$ time.
		The arrangement $A(H^=)$ of separating hyperplanes together with the $q$ facet supporting hyperplanes of $\I$ can be computed in $\O{{(m^{2}+q)}^p}$ time with the algorithm of \cite{edelsbrunner1986constructing}.
		The number of cells is bounded from above by $\O{{(m^{2}+q)}^p}$, cf.\ \cite{seipp2013adjacency}.
		Consequently, we obtain a total running time of $\O{m^2+{(m^2+q)}^{p}+{(m^2+q)}^p\cdot T_{\alg_\varepsilon}}=\O{{(m^2+q)}^{p}\cdot T_{\alg_\varepsilon}}.$
	\end{proof}
	We now focus on a single cell $c$ of the arrangement and construct the required $(1-\varepsilon)\beta$-approximation algorithm $\alg_\varepsilon$ given an oracle $\oracle_\beta$ for the non-parametric matroid $\ell$-interdiction problem.

\section{Approximation of a Single Cell}\label{sec:approx_p_l}
	In this section, we describe how to find an approximation set for a single cell $\cell$ from Algorithm~\ref{alg:framework_p_l}.
	We use several concepts from polyhedral theory, we refer to~\cite{ziegler1995LecturesPolytopes} for an in-depth introduction.
	As shown in Theorem~\ref{theo:framework_p_l}, this problem can be seen as the problem of approximating a common parametric maximization problem, with $\cell$ as its parameter set.
	However, common parametric approximation algorithms such as in~\cite{helfrich2022approximation} have rather strict requirements regarding the structure of the parameter set, and it is unclear how to transfer these algorithms into a setting where the parameter set can be a general polytope.
	Instead of using an existing approximation algorithm, we modify a \emph{Benson type} algorithm from (exact) multi-objective optimization~\cite{hamel.lohne.ea2014BensonTypeAlgorithms}, namely the dual Benson algorithm~\cite{ehrgott.lohne.ea2012DualVariantBenson} as described in~\cite{bokler.mutzel2015OutputSensitiveAlgorithmsEnumerating}.
	
	While our algorithm is tailored to Problem~\ref{prob:p_l_AX}, it is applicable to a general class\footnote{Similar to the problem class in \cite{helfrich2022approximation}, but with a more complex parameter set.} of parametric optimization problems.
	This includes both parametric minimization and maximization problems.
	
	\medskip
	
	Let $\lambda_0\in\inn(\cell)$ be an inner point of $c$.
	Note that the algorithm of \cite{edelsbrunner1986constructing} for the computation of the cells returns such an inner point for each cell.
	For every interdiction strategy $F$, we define the half-space
	\[
	\varphi(F)\coloneqq\Set{(\lambda,z)\in\R^{p+1}\colon w(B_{\lambda_0}^F,\lambda)\leq z}.
	\]
	For a $\lambda\in \cell$, the value of $w(B_{\lambda_0}^F,\lambda)$ provides a lower bound on the optimal objective value $y(\lambda)$, cf.\ Theorem~\ref{theo:framework_p_l}.
	The half-space $\varphi(F)$ can then be interpreted as a lower bound of the optimal solution value on the whole parameter set $\cell$, i.~e.~it must hold that $(\lambda,y(\lambda))\in \varphi(F)$ for every $\lambda\in\cell$.
	If we intersect all the half-spaces that correspond to an interdiction strategy and limit the parameter set to $\cell$, we get the polyhedron $P\coloneqq\{(\lambda,z):\lambda\in\cell, (\lambda,z)\in\bigcap_{F\subseteq E,\vert F\vert=\ell}\varphi(F)\}$.
	For every $\lambda\in\cell$, the point $(\lambda,y(\lambda))$ is equivalent to the point $(\lambda, \min_{(\lambda,z')\in P} z')$.
	
	\medskip
	
	The strategy of our algorithm is to maintain a set of solutions $S$ and, iteratively, to determine if $S$ is a $(1-\varepsilon)\beta$-approximation.
	In every iteration, we compute the vertices of the polyhedron $\appri=\{(\lambda,z):\lambda\in\cell, (\lambda,z)\in\bigcap_{F\in S}\varphi(F)\}$ and, for each vertex, determine the distance to the polyhedron $P$.
	Note that, by construction, $P\subseteq \appri$.
	
	As a shorthand, we denote the set of vertices of a polyhedron $\appri$ by $\ver(\appri)$.
	For every $(\lambda,z)\in \ver(\appri)$, we call $\oracle_\beta(\lambda)$, i.~e.\ the $\beta$-approximation oracle for the parameter value corresponding to this vertex.
	Let $F$ be the interdiction strategy returned, we then compare the value of $(1-\varepsilon)w(B_{\lambda_0}^{F},\lambda)$ to $z$.
	By construction of $\appri$, it holds that $z=\max _{F'\in S}w(B_{\lambda_0}^{F'},\lambda)$.
	Two cases are possible:
	First, if $(1-\varepsilon)w(B_{\lambda_0}^{F},\lambda)>z$, then $F$ is not $(1-\varepsilon)$-approximated by any strategy in $S$.
	Since $F$ itself is only a $\beta$-approximate solution, this means that the real approximation factor might be lower than  $(1-\varepsilon)\beta$.
	In this case, we stop the current iteration, add $F$ to $S$, recompute $\appri$, and start again with the vertices of the new  $\appri$.
	Conversely, in the second case $(1-\varepsilon)w(B_{\lambda_0}^{F},\lambda)\leq z$ holds, and it is ensured that $\lambda$ is already $(1-\varepsilon)\beta$-approximated in $S$.
	In this case, we do not add $F$ to $S$.
	If all vertices satisfy the second case, the convexity of $\appri$ ensures that $S$ is a $(1-\varepsilon)\beta$-approximation set (a formal proof of this fact is given in Theorem~\ref{theo:multi_approx_p_l}).
	
	\medskip
	
	As mentioned above, Algorithm~\ref{alg:approx_p_l} is a modification of the dual Benson algorithm as described in~\cite{bokler.mutzel2015OutputSensitiveAlgorithmsEnumerating}.
	The main difference is that we use an approximation oracle $\oracle_\beta$, and that we add new interdiction strategies to $S$ only if they are not already approximated by factor $(1-\varepsilon)$.
	While the original algorithm also keeps track of vertices that have already been checked to avoid redundant calls to $\oracle_\beta$, we omit this for an easier description.
	
	A formal description of the algorithm is given in Algorithm~\ref{alg:approx_p_l}.
	The first iteration is illustrated on a one-parametric example in Figure~\ref{figure::iteration}.
	
	In Theorem~\ref{theo:multi_approx_p_l}, we prove that Algorithm~\ref{alg:approx_p_l} computes a $(1-\varepsilon)\beta$-approximation set for $\cell$.
	Then, for the remainder of this section, we prove that Algorithm~\ref{alg:approx_p_l} runs in polynomial time.
	
	\begin{algorithm}[t]
		\LinesNumbered
		\Input{An instance $\Pi=(\M,w,\ell,\I)$ of Problem~\ref{prob:p_l_AX}, an $\varepsilon>0$, and a polytope $\cell\subseteq\I$.}
		\Output{A $((1-\varepsilon)\beta)$-approximation for $\Pi$ on $\cell$.}
		
		Choose an arbitrary point $\lambda_0\in\inn(\cell)$\;
		Compute a $\beta$-approximate set of $\ell$-most vital elements $F_0\algass \oracle_\beta(\lambda_0)$\;
		$S\algass \Set{F_0}$\;
		$\appri\algass\Set{(\lambda,z)\colon\lambda\in\cell,z\in\R}\cap\varphi(F_0)$\;
		\textbf{Iteration}\qquad\texttt{\textbackslash{}\textbackslash{} target for goto} \;
		\For{$(\lambda,z) \in \ver(\appri)$}{
			Compute a $\beta$-approximate set of $\ell$-most vital elements $F\algass \oracle_\beta(\lambda)$\;
			\If{$(1-\varepsilon)\cdot w(B_{\lambda_0}^{F},\lambda)>z$}
			{
				$S\algass S\cup \Set{F}$\;
				$\appri\algass \appri\cap\varphi(F)$\;
				\textbf{goto: Iteration}
			}
		}
		\Return $S$\;
		\caption{An approximation algorithm for a single cell}
		\label{alg:approx_p_l}
	\end{algorithm}
	
	\begin{theorem}\label{theo:multi_approx_p_l}
		\Cref{alg:approx_p_l} computes a $(1-\varepsilon)\beta$-approximation for the parameter set $c$.
	\end{theorem}
	\begin{proof}
		Let $S$ and $\appri$ be the sets from Algorithm~\ref{alg:approx_p_l} after termination.
		Note that the finiteness of the interdiction strategies always ensures termination.
		Since Algorithm~\ref{alg:approx_p_l} terminated, the if-condition in Line~8 must have returned false for every vertex of $\appri$.
		Therefore, for every vertex $(\lambda_v,z_v)$ of $\appri$, it holds that $z_v\geq (1-\varepsilon)\beta \cdot w(B_{\lambda_0}^{F},\lambda_v)$ for every interdiction strategy $F$.
		
		Furthermore, for any $(\lambda^*,z^*)\in \appri$ with $z^*=\min_{(\lambda,z)\in \appri} z$, it must hold that $(\lambda^*,z^*)$ is contained in a bounded face of $\appri$ and that there is an interdiction strategy $F^*\in S$ such that $z^*= w(B_{\lambda_0}^{F^*},\lambda^*)$ (the only unbounded faces of $\appri$ are the faces defined by the constraints that ensure $\lambda\in \cell$).
		
		For the sake of contradiction, assume that there is an interdiction strategy $F_C$ such that $z^* < (1-\varepsilon)\beta\cdot w(B_{\lambda_0}^{F_C},\lambda^*)$.
		Because $(\lambda^*,z^*)$ is in a bounded face of $\appri$, we can write it as a convex combination
		\[
		(\lambda^*,z^*) = \sum_{i=1}^q \mu_i\cdot(\lambda_i,z_i),
		\]
		where the points $(\lambda_i,z_i)$ for $i=1,\ldots,q$ are the vertices of this face and $\mu_i\in\mathbb{R}^q_\geq$ satisfies $\sum_{i=1}^q \mu_i=1$.
		Due to the linearity of the function $w(B_{\lambda_0}^{F^*},\lambda)$, we can also write
		\[
		w(B_{\lambda_0}^{F_C},\lambda^*)=\sum_{i=1}^q \mu_i\cdot w(B_{\lambda_0}^{F_C},\lambda_i).
		\]
		Thus, we get
		\[
		\sum_{i=1}^q \mu_i\cdot z_i < (1-\varepsilon)\beta \cdot\sum_{i=1}^q \mu_i\cdot w(B_{\lambda_0}^{F_C},\lambda_i).
		\]
		This implies that for at least one vertex $(\lambda_v,z_v)$ of $\appri$, it must hold that $z_v< (1-\varepsilon)\beta \cdot w(B_{\lambda_0}^{F_C},\lambda_v)$.
		This is a contradiction to $z_v\geq (1-\varepsilon)\beta \cdot w(B_{\lambda_0}^{F_C},\lambda_v)$.
	\end{proof}
	
	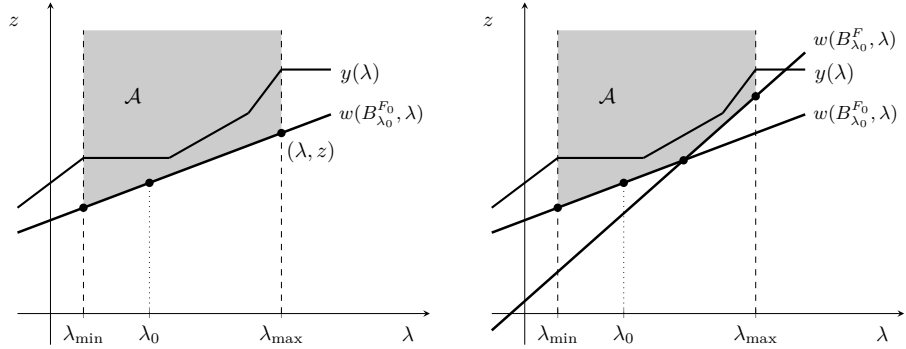
\begin{figure}
		\begin{minipage}[b]{0.48\textwidth}
			\centering
			\begin{tikzpicture}[scale=0.8,
				declare function={
					F0(\x)= (\x >=-1) * (1.5+0.2*\x);
				},
				]
				\begin{axis}[
					axis x line=middle, axis y line=middle,
					ymin=-0.5, ymax=5, ytick={},
					xmin=-1, xmax=11.5, xtick={},
					xtick={1,3,7},
					xticklabels={$\lambda_{\min}$,$\lambda_0$,$\lambda_{\max}$},
					x label style={at={(axis description cs:0.98,-0.01)}},
					xlabel=$\lambda$,
					ytick={-2},
					ylabel=$z$,
					y label style={at={(axis description cs:-0.04,0.98)}},
					domain=-1:8.5,
					]
					\addplot[fill=black!100,draw opacity=0,fill opacity=0.2] coordinates{
						(1,1.7)
						(7,2.9)
						(7,4.55)
						(1,4.55)
					};
					\addplot [line width = 1, smooth, domain=-1:1] {2.1+0.4*x} node[right] {};
					\addplot [line width = 1, smooth, domain=1:3.6] {2.5+0*x} node[right] {};
					\addplot [line width = 1, smooth, domain=3.6:6] {1.42+0.3*x} node[right] {};
					\addplot [line width = 1, smooth, domain=6:7] {-0.98+0.7*x} node[right] {};
					\addplot [line width = 1, smooth, domain=7:8.5] {3.92+0*x} node[right] {};
					\node at (9.4,3.85) (B){$y(\lambda)$};
					\addplot [very thick] {F0(x)} node [pos=1, right] {\small$w(B_{\lambda_0}^{F_0},\lambda)$};
					\addplot[mark=none, dashed] coordinates {(1, 0) (1, 4.55)};
					\addplot[mark=none, dashed] coordinates {(7, 0) (7, 4.55)};
					\addplot[mark=none, dotted] coordinates {(3, 0) (3, 2.1)};
					\node[circle,fill,inner sep=1.4] at  (3,2.1) (B){};
					\node[circle,fill,inner sep=1.4] at  (1,1.7) (B){};
					\node[circle,fill,inner sep=1.4] at  (7,2.9) (B){};
					\node at  (7.9,2.6) (B){$(\lambda,z)$};
					\node at  (2.5,3.5) (B){$\appri$};
				\end{axis}
			\end{tikzpicture}
		\end{minipage}
		\hfill
		\begin{minipage}[b]{0.48\textwidth}
			\centering
			\begin{tikzpicture}[scale=0.8,
				declare function={
					F0(\x)= (\x >=-1) * (1.5+0.2*\x);
					F(\x)= (\x >=-1) * (0.2+0.47*\x);
				},
				]
				\begin{axis}[
					axis x line=middle, axis y line=middle,
					ymin=-0.5, ymax=5, ytick={},
					xmin=-1, xmax=11.5, xtick={},
					xtick={1,3,7},
					xticklabels={$\lambda_{\min}$,$\lambda_0$,$\lambda_{\max}$},
					x label style={at={(axis description cs:0.98,-0.01)}},
					xlabel=$\lambda$,
					ytick={-2},
					ylabel=$z$,
					y label style={at={(axis description cs:-0.04,0.98)}},
					domain=-1:8.5,
					]
					\addplot[fill=black!100,draw opacity=0,fill opacity=0.2] coordinates{
						(1,1.7)
						(4.814,2.462)
						(7,3.49)
						(7,4.55)
						(1,4.55)
					};
					\addplot [line width = 1, smooth, domain=-1:1] {2.1+0.4*x} node[right] {};
					\addplot [line width = 1, smooth, domain=1:3.6] {2.5+0*x} node[right] {};
					\addplot [line width = 1, smooth, domain=3.6:6] {1.42+0.3*x} node[right] {};
					\addplot [line width = 1, smooth, domain=6:7] {-0.98+0.7*x} node[right] {};
					\addplot [line width = 1, smooth, domain=7:8.5] {3.92+0*x} node[right] {};
					\node at (9.4,3.85) (B){$y(\lambda)$};
					\addplot [very thick] {F0(x)} node [pos=1, right] {\small$w(B_{\lambda_0}^{F_0},\lambda)$};
					\addplot [very thick] {F(x)} node [pos=1, right, yshift=6] {\small$w(B_{\lambda_0}^F,\lambda)$};
					
					\addplot[mark=none, dashed] coordinates {(1, 0) (1, 4.55)};
					\addplot[mark=none, dashed] coordinates {(7, 0) (7, 4.55)};
					\addplot[mark=none, dotted] coordinates {(3, 0) (3, 2.1)};
					
					\node[circle,fill,inner sep=1.4] at  (3,2.1) (B){};
					\node[circle,fill,inner sep=1.4] at  (1,1.7) (B){};
					\node[circle,fill,inner sep=1.4] at  (4.814,2.462) (B){};
					\node[circle,fill,inner sep=1.4] at  (7,3.49) (B){};
					\node at  (2.5,3.5) (B){$\appri$};
				\end{axis}
			\end{tikzpicture}
		\end{minipage}\\[-20pt]
		\begin{minipage}[t]{0.48\textwidth}
			\subcaption{The polyhedron $\appri$ after the initialization phase: The initial interdiction strategy $F_0$ was computed by $\oracle_\beta(\lambda_0)$ and  the function $w(B_{\lambda_0}^{F_0},\lambda)$ determines a lower bound on $y(\lambda)$ on the interval $c$.\label{figure::iteration::before}}
		\end{minipage}
		\hfill
		\begin{minipage}[t]{0.48\textwidth}
			\subcaption{The polyhedron $\appri$ after the first iteration: The strategy $F$ was computed by $\oracle_\beta(\lambda_{\max})$.
				Since $(1-\varepsilon)w(B_{\lambda_0}^{F},\lambda_{\max})>w(B_{\lambda_0}^{F_0},\lambda_{\max})$, the function $w(B_{\lambda_0}^{F},\lambda)$ is used to improve the lower bound on $y(\lambda)$. \label{figure::iteration::after}}
		\end{minipage}\\[-6pt]
		\caption{Illustration of the first iteration of Algorithm~\ref{alg:approx_p_l} on a one-parametric example.
			The cell $c$ is given as the interval $[\lambda_{\min},\lambda_{\max}]$.
			Note that, by Lemma~\ref{lem:y_convex_on_cell_p_l}, the optimal interdiction value function $y$ is convex on $c$.\label{figure::iteration}}
	\end{figure}
	
	We now analyze the running time complexity of Algorithm~\ref{alg:approx_p_l}.
	By construction, $\cell$ satisfies some useful geometric properties.
	It is a polytope that is the intersection of at most $m^2$ half-spaces.
	Furthermore, the encoding length of each of these half-spaces is bounded by a polynomial in the encoding length of the matroid $\mathcal{M}$.
	Similar, the encoding length of each half-space $\varphi(F)$ that is additionally constructed in Algorithm~\ref{alg:approx_p_l} is also bounded by such a polynomial.
	
	As a consequence, any vertex $(\lambda,z)$ of $P$ that is encountered in Algorithm~\ref{alg:approx_p_l} as well as $\lambda_0$ can be encoded with an encoding length polynomial in the encoding length of $\mathcal{M}$, cf.\ \cite{grotschel.lovasz.ea1988RationalPolyhedra}.
	Therefore, throughout Algorithm~\ref{alg:approx_p_l}, all arithmetic operations and calls to $\oracle_\beta$ run in polynomial time (cf.\ the analysis of the dual Benson algorithm in~\cite{bokler.mutzel2015OutputSensitiveAlgorithmsEnumerating}).
	By $T_{\oracle_\beta}$, we denote the polynomial running time bound for a call to $\oracle_\beta$ implied by this fact.
	
	\medskip
	
	There are exactly $\left|S\right|$ iterations (i.~e.\ times we jump from Line~11 to Line~5.) in Algorithm~\ref{alg:approx_p_l}, and the running time of each single iteration can be bounded further.
	Let $v=\left|S\right|+m^2$.
	Enumerating the vertices at the beginning of an iteration can be done in $\mathcal{O}(v\log v+ v^{\lfloor\nicefrac{p}{2} \rfloor})$ with the algorithm from~\cite{chazelle1993OptimalConvexHull}.
	There are at most $\mathcal{O}(v^{\lfloor\nicefrac{p}{2} \rfloor})$ vertices, and for each vertex there is at most one call to $\oracle_\beta$.
	Put together, we get a running time in
	\[
	\mathcal{O}\left(\left|S\right|\left( v\log v+ v^{\lfloor\nicefrac{p}{2}\rfloor}T_{\oracle_\beta}\right)\right).
	\]
	If we can show that $\left|S\right|$ is bounded by a polynomial in the encoding length of $\mathcal{M}$, then the entire running time of Algorithm~\ref{alg:approx_p_l} is also bounded by such a polynomial.
	
	For this, we need to define two additional notations:
	
	First, for an interdiction strategy $F$, we define the ($p+1$)-dimensional projection of the weight of the interdicted basis $B^F_{\lambda_0}$ as
	\[
	\gamma(F)\coloneqq \left(\sum_{e\in B_{\lambda_0}^F}b_{1,e},\ldots,\sum_{e\in B_{\lambda_0}^F}b_{p,e},\sum_{e\in B_{\lambda_0}^F}a_e,\right).
	\]
	In particular, 
	\begin{equation}\label{equation::inParticular}
		w(B_{\lambda_0}^F,\lambda)=\sum_{i=1}^p\lambda_i \gamma_i(F) + \gamma_{p+1}(F).
	\end{equation}
	
	Second, we define a \emph{$(1-\varepsilon)$-hyperrectangle} for a point $r\in \mathbb{R}^{p+1}_>$ as the set
	\[
	R(r)=\Set{x\in\mathbb{R}^{p+1}: r\geq x \geq (1-\varepsilon)r}.
	\]
	
	In~\cite{papadimitriou.yannakakis2000ApproximabilityTradeoffsOptimal}, it is shown how to subdivide the $\mathbb{R}_>^{p+1}$ into a polynomial number of $(1-\varepsilon)$-rectangles such that for every interdiction strategy $F\subseteq E$ with $\vert F\vert=\ell$ the corresponding point $\gamma(F)$ is in at least one $(1-\varepsilon)$-rectangle.
	The cardinality of this set of rectangles is bounded in $\mathcal{O}\left(\left(\nicefrac{\rho}{\varepsilon}\right)^{p+1}\right)$, where $\rho$ denotes an upper bound on the binary encoding length of the components that can appear in a vector $\gamma(F)$.
	Since every such component consists of a sum of weights, $\rho$ is bounded by a polynomial in the encoding length of the matroid interdiction problem.
	
	Additionally, for every $r\in \mathbb{R}^{p+1}_>$, we can show the following result:
	\begin{lemma}\label{lemma:grid_once}
		Let $r\in \mathbb{R}^{p+1}_>$.
		At the end of Algorithm~\ref{alg:approx_p_l}, there is at most one interdiction strategy $F\in S$ where the point $\gamma(F)$ is in $R(r)$.
	\end{lemma}
	\begin{proof}
		We prove by contradiction.
		Let $F_C\in S$ be a strategy that was added to $S$ later than $F$ during Algorithm~\ref{alg:approx_p_l}.
		Now, let $(\lambda,z)$ be the vertex of $\appri$ that caused $F_C$ being added to $S$, i.~e.\ $\oracle_\beta$ was called with $\lambda$, returned $F_C$, and $(1-\varepsilon)\cdot w(B_{\lambda_0}^{F_C},\lambda)>z$ (see Steps~$7$ to~$9$ in Algorithm~\ref{alg:approx_p_l}).
		At the same time, $w(B_{\lambda_0}^{F},\lambda)\leq z$ must hold because $F$ was already in $S$ at that point in Algorithm~\ref{alg:approx_p_l}.
		Therefore, it holds $w(B_{\lambda_0}^{F},\lambda)<(1-\varepsilon)\cdot w(B_{\lambda_0}^{F_C},\lambda)$.
		
		\medskip
		
		By construction of $R$, the point $\gamma(F)$ being in $R$ implies that $0\leq(1-\varepsilon)r\leq\gamma(F)\leq r$.
		Since $c\subseteq\mathbb{R}_\geq^p$ and, thus, $\lambda\geq0$, this implies (together with Equation~\ref{equation::inParticular}) that
		\[
		(1-\varepsilon)\left(\sum_{i=1}^{p}\lambda_i r_{i} + r_{p+1} \right) \leq w(B_{\lambda_0}^{F},\lambda)\leq \sum_{i=1}^{p}\lambda_i r_{i} + r_{p+1}.
		\]
		For the sake of contradiction, let us now assume that the point $\gamma(F_C)$ is also in $R$.
		For the same reason as above, it holds that
		\[
		(1-\varepsilon)\left(\sum_{i=1}^{p}\lambda_i r_{i} + r_{p+1} \right) \leq w(B_{\lambda_0}^{F_C},\lambda)\leq \sum_{i=1}^{p}\lambda_i r_{i} +  r_{p+1}.
		\]
		At the same time, $w(B_{\lambda_0}^{F},\lambda)<(1-\varepsilon) w(B_{\lambda_0}^{F_C},\lambda)$ holds, which implies 
		\begin{alignat*}{2}
			&& (1-\varepsilon)\left(\sum_{i=1}^{p}\lambda_i r_{i} + r_{p+1} \right) & < (1-\varepsilon) w(B_{\lambda_0}^{F_C},\lambda) \\
			\implies\quad&& \sum_{i=1}^{p}\lambda_i r_{i} + r_{p+1} & <  w(B_{\lambda_0}^{F_C},\lambda).
		\end{alignat*}
		This is a contradiction.
	\end{proof}
	
	By combining Lemma~\ref{lemma:grid_once} with the cardinality result from \cite{papadimitriou.yannakakis2000ApproximabilityTradeoffsOptimal}, it follows that there can be at most $\mathcal{O}\left(\left(\nicefrac{\rho}{\varepsilon}\right)^{p+1}\right)$ solutions in $S$.
	Therefore, the polynomial running time of Algorithm~\ref{alg:approx_p_l} follows directly.
	\begin{theorem}
		For fixed $p$, the running time of Algorithm~\ref{alg:approx_p_l} is bounded by a polynomial in $\varepsilon^{-1}$, the encoding length of the matroid interdiction problem instance $\Pi$, and the running time of the oracle $\oracle_\beta$.
	\end{theorem}

\section{Applications}\label{sec:application_p_l}
	Given an oracle $\oracle_\beta$ for the non-parametric matroid $\ell$-interdiction problem, the approximation framework from Section~\ref{sec:framework_p_l} together with the $(1-\varepsilon)\beta$-approximation from Section~\ref{sec:approx_p_l} yields the first (polynomial time) $(1-\varepsilon)\beta$-approximation algorithm for the multi-parametric matroid $\ell$-interdiction problem.
	In this section, we apply this approximation algorithm to two prominent types of matroids:
	For \emph{partition} matroids we obtain the first FPTAS and for \emph{graphic} matroids we obtain the first $(1-\varepsilon)\frac{1}{4}$-approximation.
	
	\textbf{Partition matroids.}
	If $\M$ is a partition matroid, we are given a natural number $r\in\N_{>0}$ and a partition $E=E_1\dot\cup\ldots\dot\cup E_r$ of the matroid's ground set $E$.
	Further, there exists a natural number $u_i\in\N_{>0}$ for each $i=1,\dotsc,r$.
	The independent sets are given by $\F=\Set{R\subseteq E\colon \vert R\cap E_i\vert\leq u_i\; \forall i=1,\dotsc,r}$.
	The bases are those subsets that share exactly $u_i$ elements with the set $E_i$ for each $i=1,\dotsc,r$.
	
	For fixed $\lambda\in\I$, the partition matroid $\ell$-interdiction problem can be solved exactly in polynomial time $\O{m^3\ell}$ using dynamic programming, cf.\ \cite{weninger2024interdiction}.
	We use their algorithm as oracle $\oracle_\beta$ with $\beta=1$ to obtain an FPTAS for the multi-parametric partition matroid $\ell$-interdiction problem.
	\begin{corollary}
		Let $0<\varepsilon<1$.
		Algorithm~\ref{alg:framework_p_l} together with Algorithm~\ref{alg:approx_p_l} yields an FPTAS for the multi-parametric partition matroid $\ell$-interdiction problem.
	\end{corollary}
	
	\textbf{Graphic matroids.}
	If $\M$ is a graphic matroid, we obtain the multi-parametric minimum spanning tree $\ell$-interdiction problem.
	Let $G=(V,E)$ be an undirected graph with edge set $E$ and node set $V$.
	By Assumption \ref{ass:a_1_l_AX}, we can assume that $G$ is $(\ell+1)$-edge-connected, i.~e. $G$ remains connected after the removal of $\ell$ arbitrary edges.
	The set $\F$ of independent sets is given by the acyclic subsets of $E$, i.~e by the trees.
	The bases correspond to the spanning trees of $G$ and a set of $\ell$-most vital elements corresponds to a set of $\ell$-most vital edges.

	Each edge is associated with a multi-parametric weight
	$$w(e,\lambda)\coloneqq w(e,\lambda_1,\dotsc,\lambda_p)= a_e + \lambda_1 b_{1,e} +\dotsc+\lambda_p b_{p,e},$$ where $a_e, b_{1,e},\dotsc, b_{p,e}\in \mathbb{Q}$.
	In the multi-parametric minimum spanning tree $\ell$-interdiction problem, the goal is to compute, for each possible parameter vector, a set of $\ell$-most vital edges and the weight of an optimally $\ell$-interdicted spanning tree $y(\lambda)$.
	For fixed $\ell\in\N_{>0}$, the problem can be solved in polynomial time, cf.\ \cite{hausbrandt2024theparametric}.
	If $\ell$ is part of the input of the problem, the problem becomes $\mathsf{NP}$-hard for a fixed parameter vector $\lambda \in\I$, cf.\ \cite{hausbrandt2024theparametric}.
	
	For a fixed $\lambda\in\I$, \cite{zenklusen2015approximation} developed a $\nicefrac{1}{14}$-approximation for the minimum spanning tree interdiction problem.
	\cite{linhares2017improved} improved upon the approximation guarantee to obtain a $\nicefrac{1}{4}$-approximation.
	Using their $\nicefrac{1}{4}$-approximation algorithm as oracle $\oracle_\beta$ with $\beta=\nicefrac{1}{4}$ for our $(1-\varepsilon)\beta$-approximation algorithm, we obtain the following corollary.
	\begin{corollary}
		Let $0<\varepsilon<1$.
		Algorithm~\ref{alg:framework_p_l} together with Algorithm~\ref{alg:approx_p_l} yields a $(1-\varepsilon)\frac{1}{4}$-approximation algorithm for the multi-parametric minimum spanning tree $\ell$-interdiction problem.
	\end{corollary}
	
\section{Conclusion}
	In this article, we investigated a matroid interdiction problem, where the interdiction strategies correspond to the subsets of the matroid's ground set of a given cardinality.
	An optimal interdiction strategy maximizes the weight of a minimum weight basis when applied.
	We considered a multi-parametric version of this $\mathsf{NP}$-hard problem, in which each element is assigned a weight that depends linearly on a parameter vector of fixed size.
	The goal is to compute an optimal interdiction strategy and the weight of an optimal interdicted minimum weight basis for each possible parameter vector from a given parameter polytope.
	We showed that the problem can be approximated in polynomial time whenever there exists an approximation algorithm for the corresponding non-parametric problem, i.~e.\ the problem obtained when fixing the parameter vector.
	We developed an algorithm with approximation guarantee arbitrary close to that of an algorithm for the non-parametric problem.
	We obtained an FPTAS for the multi-parametric partition matroid interdiction problem and a $(1-\varepsilon)\frac{1}{4}$-approximation for the multi-parametric minimum spanning tree interdiction problem.
	As a byproduct, we developed the first approximation algorithm for a broad class of common multi-parametric optimization problems that works on an arbitrary parameter polytope.	
	
	\section*{Acknowledgments}
	This work was partially supported by the project "Ageing Smart -- Räume intelligent gestalten" funded by the Carl Zeiss Foundation and the DFG grant RU 1524/8-1, project number 508981269 and GRK 2982, 516090167 ``Mathematics of Interdisciplinary Multiobjective Optimization''.
	
	\bibliography{bib_matroid}
	
\end{document}